\documentclass{amsart}
\usepackage{amscd,amsmath,amssymb,amsfonts}
\usepackage[cmtip, all]{xy}

\newtheorem{thm}{Theorem}[section]
\newtheorem{prop}[thm]{Proposition}
\newtheorem{lem}[thm]{Lemma}

\renewcommand{\theclaim}{\kern-3pt}

\theoremstyle{definition}
\newtheorem{definition}[thm]{Definition}

\theoremstyle{remark}
\newtheorem{rem}[thm]{Remark}
\newtheorem{rems}[thm]{Remarks}
\newtheorem{ex}[thm]{Example}

\numberwithin{equation}{section}

\newcommand{\sH}{{\mathcal H}}

\newcommand{\sL}{{\mathcal L}}
\newcommand{\sM}{{\mathcal M}}

\newcommand{\sO}{{\mathcal O}}

\newcommand{\sS}{{\mathcal S}}

\newcommand{\A}{{\mathbb A}}

\newcommand{\C}{{\mathbb C}}

\newcommand{\G}{{\mathbb G}}

\renewcommand{\P}{{\mathbb P}}

\newcommand{\X}{{\mathbb X}}

\newcommand{\Z}{{\mathbb Z}}

\renewcommand{\L}{{\mathbb L}}

\renewcommand{\phi}{\varphi}

\newcommand{\codim}{{\rm codim}}

\newcommand{\Pic}{{\rm Pic}}
\newcommand{\Div}{{\rm div}}
\newcommand{\Hom}{{\rm Hom}}

\newcommand{\Spec}{{\rm Spec \,}}

\newcommand{\sHom}{{\mathcal{H}{om}}}

\newcommand{\id}{{\operatorname{id}}}

\newcommand{\Sch}{{\operatorname{\mathbf{Sch}}}}

\newcommand{\op}{{\text{\rm op}}}

\newcommand{\del}{\partial}

\newcommand{\Spt}{{\mathbf{Spt}}}
\newcommand{\Spc}{{\mathbf{Spc}}}
\newcommand{\Sm}{{\mathbf{Sm}}}

\renewcommand{\lim}{\operatornamewithlimits{\varprojlim}}
\newcommand{\colim}{\operatornamewithlimits{\varinjlim}}

 \newcommand{\Ab}{{\mathbf{Ab}}}

\newcommand{\HZ}{{\operatorname{\sH \Z}}}

\newcommand{\SH}{{\operatorname{\sS\sH}}}

\newcommand{\eff}{{\mathop{eff}}}

\newcommand{\Nis}{{\operatorname{Nis}}}

\newcommand{\ds}{{/\kern-3pt/}}

\newcommand{\Th}{{\mathop{\rm{Th}}}}

\newcommand{\Gr}{{\mathop{\rm{Grass}}}}

\newcommand{\SP}{{\mathbf{SP}}}

\newcommand{\MGL}{{\operatorname{MGL}}}
\newcommand{\BGL}{{\operatorname{BGL}}}
\newcommand{\MU}{{\operatorname{MU}}}

\newcommand{\lci}{l.c.i.\ }
\newcommand{\Cone}{{\operatorname{Cone}}}

\begin{document}

\title{Comparison of cobordism theories}
\author{Marc Levine}
\address{
Department of Mathematics\\
Northeastern University\\
Boston, MA 02115\\
USA}
\email{marc@neu.edu}


\thanks{The  author gratefully acknowledges the support of the Humboldt Foundation, and the support of the NSF via the grant DMS-0457195}
\renewcommand{\abstractname}{Abstract}
\begin{abstract}  Relying on results of Hopkins-Morel, we show that, for $X$ a quasi-projective variety over a field of characteristic zero, the canonical map $\Omega_n(X)\to \MGL'_{2n,n}(X)$ is an isomorphism.
\end{abstract}
\maketitle
\tableofcontents

\section*{Introduction}  In what follows, $k$ will be a fixed base-field of characteristic zero. We let $\Sm/k$ denote the category of smooth, quasi-projective varieties over $k$, and $\SH(k)$ the Morel-Voevodsky motivic stable homotopy category (see \cite{Morel, Voevodsky}). We denote the classical stable homotopy category by $\SH$.

Voevodsky \cite{Voevodsky} has defined the bigraded cohomology theory $\MGL^{*,*}$ on $\Sm/k$ as the theory represented by the algebraic Thom complex $\MGL\in\SH(k)$. This in turn is the $T$-spectrum constructed as the algebraic analog of the classical Thom complex $\MU\in\SH$, replacing $BU_n$ with $\BGL_n$ and the Thom space of the universal $\C^n$-bundle $E_n\to BU_n$ with the Thom space of the universal rank $n$ vector bundle over $\BGL_n$. 

Besides the definition of complex cobordism via stable homotopy theory, one can also describe $\MU^*(X)$ as cobordism classes of $\C$-oriented proper maps. As an algebraic version of this construction, we have defined with F. Morel the theory $X\mapsto \Omega^*(X)$, also called algebraic cobordism, and show that $\Omega^*$ is the universal oriented cohomology theory on $\Sm/k$, in the sense of \cite[definition 5.1.3]{LevineMorel}. This universal property gives us a natural transformation of functors on $\Sm/k$:
\[
\vartheta^\MGL(X):\Omega^*(X)\to \MGL^{2*,*}(X).
\]
As remarked in the introduction to \cite{LevineMorel}, Hopkins-Morel have constructed a spectral sequence
\[
E^{p,q}_2(n)=H^{p-q,n-q}(F)\otimes\L^{q}\Longrightarrow \MGL^{p+q,n}(F);
\]
using this, it is easy to show that $\vartheta^\MGL(k)$ is surjective. Choosing an embedding $\sigma:k\to\C$, gives the natural transformation
\[
\vartheta^{\MU,\sigma}(X):\Omega^*(X)\to \MU^{2*}(X(\C))
\]
and the commutative diagram
\[
\xymatrixcolsep{50pt}
\xymatrix{
\Omega^*(X)\ar[r]^{\vartheta^\MGL(X)}\ar[dr]_{\vartheta^{\MU,\sigma}}& \MGL^{2*,*}(X)\ar[d]^{\text{Re}^\sigma(X)}\\
&\MU^{2*}(X(\C)).
}
\]
As $\vartheta^{\MU,\sigma}(k)$ is an isomorphism (\cite[theorem 1.2.7]{LevineMorel}), this shows that 
$\vartheta^\MGL(k)$ is an isomorphism as well. The purpose of this paper is to extend this to show that 
$\vartheta^\MGL(X)$ is an isomorphism for all $X\in\Sm/k$.

In fact, we prove more. In \cite{LevineOrient}, we have shown how one can extend $\MGL^{*,*}$ to a {\em bi-graded oriented duality theory} $(\MGL'_{*,*}, \MGL^{*,*})$ on quasi-projective $k$-schemes, $\Sch_k$, and how $\vartheta^\MGL$ extends to a natural transformation
\[
\vartheta_{\MGL'}:\Omega_*\to \MGL'_{2*,*}.
\]
Here $\Omega_*$ is the extension of $X\mapsto \Omega^{\dim X-*}(X)$ to an oriented Borel-Moore homology theory on $\Sch_k$, as described in \cite{LevineMorel}. Our main result in this paper is that 
$\vartheta_{\MGL'}(X)$ is an isomorphism for all $X\in\Sch_k$.

We review the underlying foundations of motivic homotopy theory in \S\ref{sec:MotHom}, recall the construction of $\MGL$-theory and its extension to a Borel-Moore homology theory in \S\ref{sec:BMHomology} and discuss the geometric theory $\Omega_*$ and its relation to $\MGL$-theory in \S\ref{sec:OmegaEtc}, stating our main result in theorem~\ref{thm:main}. The next three sections 
deal with the proof of theorem~\ref{thm:main}. In \S\ref{sec:Fields}, we use the Hopkins-Morel spectral sequence to get generators for $\MGL^{2n,n}(F)$ and $\MGL^{2n-1,n}(F)$, for $F$ a finitely generated field over $k$. As outlined above, this proves the result for $X=\Spec F$. In \S\ref{sec:Boundary}, we look at the boundary map
\[
\del:\MGL_{2n+1,n}'(k(X))\to\colim_{W\subset X}\MGL'_{2n,n}(W)
\]
and give a formula for $\del$ in terms of the divisor classes defined in \cite{LevineMorel}. Combining this information with the right-exact localization sequence for $\Omega_*$ gives a proof of the main theorem by using induction on the dimension of $X$. 

Although the existence of the Hopkins-Morel spectral sequence has been announced some time ago, and its construction has been described in lectures (e.g., lectures of M. Hopkins at Harvard in the Spring of 2006), the details of the construction have not yet been published. The cautious reader may therefore want to consider the results of this paper as conditional, relying on the existence of the Hopkins-Morel spectral sequence.

\section{Motivic homotopy theory}\label{sec:MotHom} We begin by recalling some basic notions in the 
Morel-Voevodsky motivic stable homotopy category $\SH(k)$; we refer the reader to \cite{Morel, Voevodsky} for details. One starts with the category $\Spc(k)$ of {\em spaces over $k$}, this being the category of presheaves of simplicial sets on $\Sm/k$. The pointed version, $\Spc_\bullet(k)$ is the category of presheaves of pointed simplicial sets on $\Sm/k$. 

The standard operations on   simplicial sets, e.g., products $A\times_CB$, co-products $A\cup_CB$, quotients $A/B:=A\cup_B pt$, etc., are all inherited by $\Spc(k)$, by operating on the values of the given presheaves. In particular, one has wedge product and pointed union in $\Spc_\bullet(k)$. In addition, the internal Hom in simplicial sets
\[
\sHom(A,B)(n):= \Hom_\Spc(A\times\Delta^n,B)
\]
gives rise to an internal Hom in $\Spc(k)$; we have pointed versions as well.

Letting $\Spc$ denote the category of simplicial sets (and $\Spc_\bullet$ pointed simplicial sets), taking the constant presheaf gives functors $\Spc\to \Spc(k)$, $\Spc_\bullet\to\Spc_\bullet(k)$. We also have the fully faithful functor $\Sm/k\to \Spc(k)$, taking $X\in\Sm/k$ to the representable presheaf (of sets) $\Hom_{\Sm/k}(-,X)$, where we identify a set $S$ with the constant simplicial set $n\mapsto S$. Similarly, given $X\in\Sm/k$ and a $k$-point $x\in X$, we may consider the pointed scheme $(X,x)$ as an object in $\Spc_\bullet(k)$.

\begin{ex} The {\em simplicial suspension operator} $\Sigma_s:\Spc_\bullet(k)\to \Spc_\bullet(k)$ is $\Sigma_sW:=S^1\wedge W$. Similarly, define $\Sigma_{gm}W:=\G_m\wedge W$, where $\G_m:=\A^1\setminus\{0\}$, pointed by $1$.

For $a\ge b$, define the {\em weighted sphere} $S^{a,b}\in\Spc_\bullet(k)$ by
\[
S^{a,b}:=\Sigma_s^{a-b}\Sigma_{gm}^b S^0.
\]

We have as well the $T$-suspension operator
\[
\Sigma_TW:=(\A^1/(\A^1\setminus\{0\}))\wedge W
\]
\end{ex}

The {\em unstable motivic homotopy category} $\sH(k)$ is formed from $\Spc(k)$ by a two-step localization process. First, one introduces the Nisnevich topology. Given $A\in \Spc(k)$, and 
 $x\in X\in \Sm/k$, one has the stalk of $A$ at $x$, 
 \[
 A_x:=\colim_{x\in U\to X}A(U)
 \]
 as $x\in U\to X$ runs over all Nisnevich neighborhoods of $x$ in $X$, i.e., over all \'etale maps $f:U\to X$ together with a lifting of the inclusion $x\to X$ to $x\to U$. Declare a map $f:A\to B$ in $\Spc(k)$ to a {\em Nisnevich local weak equivalence} if for each $x\in X\in \Sm/k$, the map of simplicial sets 
 $f_x:A_x\to B_x$ induces an isomorphism on $\pi_0$ and on all homotopy groups $\pi_n(A_x,q)\to \pi_n(B_x,f_x(q))$, $n\ge1$. Inverting all Nisnevich local weak equivalences forms the category $\sH_\Nis(k)$. 
 
 Next, one defines an object $Z$ of $\sH_\Nis(k)$ to be {\em $\A^1$-local} if, for all $X\in\Sm/k$, and all $n\ge0$,  the map
 \[
 p^*:\Hom_{\sH_\Nis(k)}(\Sigma^n_sX_+,Z)\to \sHom(\Sigma^n_sX\times\A^1_+,Z)
 \]
 is an isomorphism in $\sH_\Nis(k)$.   A $f:A\to B$ in $\sH_\Nis(k)$ is an {\em $\A^1$-weak equivalence} if 
 \[
 p^*:\Hom_{\sH_\Nis(k)}(\Sigma^n_sB_+,Z)\to :\Hom_{\sH_\Nis(k)}(\Sigma^n_sA_+,Z)
 \]
 is an isomorphism for all $\A^1$-local $Z$ and all $n\ge0$. Inverting all the $\A^1$-weak equivalences in $\sH_\Nis(k)$ gives us the category $\sH(k)$. The same construction in the pointed setting gives us the pointed version $\sH_\bullet(k)$.
 
\begin{rem}  In $\sH_\bullet(k)$, we have isomorphisms
\[
(\P^1,\infty)\cong \A^1/(\A^1\setminus\{0\})\cong S^1\wedge\G_m=S^{2,1}.
\]
Thus, the $T$-suspension operator $\Sigma_T$ is isomorphic (in $\sH(k)$) to $\P^1\wedge-$.
\end{rem}

For each object $A$ in $\sH_\bullet(k)$ we have the {\em bi-graded $\A^1$ homotopy sheaves} $\pi^{\A^1}_{a,b}(A)$, defined (for $a\ge b\ge0$) as the Nisnevich sheafification of the presheaf
\[
U\mapsto \Hom_{\sH_\bullet(k)}(U_+\wedge S^{a,b}, A)
\]

We now pass to the stable theory. The {\em motivic stable homotopy category over $k$}, $\SH(k)$, is defined as a localization of the category $\Spt(k)$ of {\em $T$-spectra}: objects in $\Spt(k)$ are sequences
 \[
 E:=(E_0,E_1,\ldots)
 \]
 $E_n\in\Spc_\bullet(k)$,  together with bonding maps
 \[
 \epsilon_n:E_n\wedge T\to E_{n+1}.
 \]
 A morphism $f:E\to E'$ is a sequence of maps $f_n:E_n\to E_n'$ in $\Spc(k)$ that commute with the respective bonding maps. 
 
 The construction of $\Spt(k)$ and $\SH(k)$ models that of the category of spectra, $\Spt$ and the classical stable homotopy category $\SH$:\\
\\ 
1. The category of spectra, $\Spt$, is the category with objects sequences of pointed simplicial sets, $(E_0, E_1,\ldots)$, together with bonding maps $E_n\wedge S^1\to E_{n+1}$, where a morphism $E\to E'$ is given by a sequence of maps $E_n\to E_n'$ commuting with the bonding maps.  
 
 Sending a pointed simplicial set $S\in\Spc_\bullet$ to 
 \[
 \Sigma^\infty S:=(S, S\wedge S^1, S\wedge S^1\wedge S^1,\ldots)
 \]
 with identity bonding maps defines a functor $\Sigma^\infty:\Spc_\bullet\to\Spt$. Defining the suspension operator $\Sigma$ on $\Spt$ by
 \[
 \Sigma(E_0,E_1,\ldots):=(\Sigma E_0, \Sigma E_1,\ldots),
 \]
 the functor $\Sigma^\infty$ commutes with the respective suspension operators.\\
 \\
 2. Given a morphism $f:E\to E'$ in $\Spt$, one has the {\em cofiber sequence}
 \[
 E\xrightarrow{f}E'\xrightarrow{i}C(f)\xrightarrow{p} \Sigma E
 \]
 where $C(f)$ is the sequence of mapping cones 
 \[
 C(f_n):=([0,1]\times E_n/0\times E_n\cup [0,1]\times *)\cup_{1\times E_n}E'_n.
 \]
 The map $i$ is the sequence of inclusions $E'_n\to C(f_n)$ and $p$ is the sequence of quotient maps $C(f_n)\to S^1\wedge E_n$ (collapsing $i(E'_n)$).  \\
\\
3. For a spectrum $E=(E_0,E_1,\ldots)$, the {\em stable homotopy group} $\pi_a^s(E)$ is defined by
\[
\pi_a^s(E):=\colim_n\pi_{n+a}E_n
\]
where the colimit is taken using the maps sending $f:S^{n+a}\to E_n$ to
\[
S^{n+a+1}= S^{n+a}\wedge S^1\xrightarrow{f\wedge \id} E_n\wedge S^1\to E_{n+1}.
\]
Note that $\pi^s_a(E)$ is defined for all $a\in\Z$.
A morphism $E\to E'$ is a {\em stable weak equivalence} if the induced map $\pi^s_a(E)\to \pi^s_a(E')$ is an isomorphism for all $a$.\\
\\
4. The stable homotopy category $\SH$ is formed from $\Spt$ by inverting the stable weak equivalences; the suspension operator on $\Spt$ descends to one on $\SH$. The functor
\[
\Sigma^\infty:\Spc_\bullet\to\Spt
\]
descends to 
\[
\Sigma^\infty:\sH\to \SH,
\]
commuting with the respective suspension operators, and the suspension operator $\Sigma$ on $\SH$ is an equivalence. $\SH$ is a triangulated category with translation $\Sigma$ and distinguished triangles given by the images of cofiber sequences (up to isomorphism).\\
\\ 
The construction of the stable homotopy category $\SH$, suitably modified, gives us the motivic stable homotopy category $\SH(k)$. We give here a quick sketch of the construction.

 Given $E=(E_0, E_1,\ldots)$ in $\Spt(k)$, the bonding maps give rise to the inductive system
 \[
\ldots\to \pi^{\A^1}_{2n+a,n+b}(E_n)\to \pi^{\A^1}_{2n+2+a,n+1+b}(E_{n+1})\to\ldots
 \]
 defined by sending $f:U_+\wedge S^{2n+a,n+b}\to E_n$ to the composition
 \begin{multline*}
U_+\wedge S^{2n+2+a,n+1+b}  =U_+\wedge S^{2n+a,n+b}\wedge S^{2,1}
\\\xrightarrow{f\wedge \id}
E_n\wedge S^{2,1}
\cong E_n\wedge T\xrightarrow{\epsilon_n}E_{n+1}.
\end{multline*}
Define the {\em motivic stable homotopy sheaf} $\pi^{\A^1}_{a,b}(E)$ by
\[
\pi^{\A^1}_{a,b}(E):=\colim_n\pi^{\A^1}_{2n+a,n+b}(E_n).
\]
Note that $\pi^{\A^1}_{a,b}(E)$ is defined for all $a,b\in\Z$.

A morphism $f:E\to E'$ in $\Spt(k)$ is a {\em  stable $\A^1$ weak equivalence} if $f$ induces an isomorphism
\[
f_*:\pi^{\A^1}_{a,b}(E)\to \pi^{\A^1}_{a,b}(E')
\]
for all $a,b$. $\SH(k)$ is formed from $\Spt(k)$ by inverting all $\A^1$ stable weak equivalences.

\begin{rems}  1. We have the {\em infinite $T$-suspension functor}
\[
\Sigma^\infty_T:\Spc_\bullet(k)\to \Spt(k)
\]
sending $A$ to the sequence $(A, A\wedge T, A\wedge T\wedge T,\ldots)$. We have as well the operations $\Sigma_T$, $\Sigma_s$, $\Sigma_{gm}$ on $\Spt(k)$, defined by
\[
\Sigma_?(E_0,E_1,\ldots):=(\Sigma_?E_0, \Sigma_?E_1,\ldots),
\]
and commuting with $\Sigma^\infty_T$.\\
\\
2. $\Sigma_T^\infty$ descends to a functor
\[
\Sigma^\infty_T:\sH_\bullet(k)\to \SH(k)
\]
and $\Sigma_T$, $\Sigma_s$ and $\Sigma_{gm}$ descend to operators  on $\SH(k)$ and $\sH_\bullet(k)$, with
\[
\Sigma_T\cong\Sigma_s\circ\Sigma_{gm}\cong \Sigma_{gm}\circ\Sigma_s.
\]
\ \\
3. $\Sigma_T$ is invertible on $\SH(k)$, the inverse given by
\[
(E_0, E_1,\ldots)\mapsto(pt, E_0, E_1,\ldots)
\]
Thus $\Sigma_s$ and $\Sigma_{gm}$ are also invertible on $\SH(k)$. For $a,b\in\Z$, define the operator $\Sigma^{a,b}$ on $\SH(k)$ by
\[
\Sigma^{a,b}:=\Sigma_s^{a-b}\circ\Sigma_{gm}^b.
\]
\ \\
4. Given a morphism $f:E\to F$ in $\Spc_\bullet(k)$, we have the cone
\[
\Cone(f):=([0,1]\times E/0\times E\cup [0,1]\times *)\cup_{1\times E}F,
\]
and the  cofiber sequence
\[
E\xrightarrow{f} F\to \Cone(f)\to\Sigma_sE
\]
defined just as for spaces. Given a morphism $f:E\to F$ in $\Spt(k)$, we have the cone
\[
\Cone(f)_n=\Cone(f_n:E_n\to F_n)
 \]
 and the cofiber sequence
\[
E\xrightarrow{f} F\to \Cone(f)\to\Sigma_sE
\]
which is just the sequence of cofiber sequences
\[
E_n\xrightarrow{f_n} F\to \Cone(f_n)\to\Sigma_sE_n.
\]
$\SH(k)$ is a triangulated category, with translation functor $\Sigma_s$. The distinguished triangles are those isomorphic to the image of a cofiber sequence in $\Spt(k)$.
\end{rems}

\begin{definition} Let $E$ and $F$ be in $\SH(k)$. The {\em $E$-cohomology groups of $F$}, $E^{a,b}(F)$, are
\[
E^{a,b}(F):=\Hom_{\SH(k)}(F, \Sigma^{a,b}E).
\]
For $X\in\Sm/k$, define
\[
E^{a,b}(X):E^{a,b}(\Sigma^\infty_TX_+)=\Hom_{\SH(k)}(\Sigma^\infty_TX_+, \Sigma^{a,b}E).
\]
\end{definition}

\section{$\MGL$ cohomology and Borel-Moore homology} \label{sec:BMHomology}

Let $p:U\to B$ be a vector bundle over a scheme $B$ with 0 section $0_B$. The {\em Thom space} of $U$ is the space over $k$
\[
\Th(U):=U/U\setminus0_B\in \Spc_\bullet(k).
\]
There is a canonical isomorphism
\[
\Th(U\oplus\sO_B)\cong \Th(U)\wedge T.
\]

 We recall the Morel-Voevodsky theory of algebraic cobordism $X\mapsto \MGL^{**}(X)$, represented in the Morel-Voevodsky motivic stable homotopy category $\SH(k)$ by the algebraic version, $\MGL$, of the classical Thom spectrum $\MU$. As in classical topology, we have
\[
\MGL=(\MGL_0,\MGL_1,\ldots,\MGL_n,\ldots)
\]
where $\MGL_n$ is the Thom space of the universal rank $n$ quotient bundle $U_n$ over $\BGL_n$:
\[
\MGL_n:=\Th(U_n\to \BGL_n)
\]
$\BGL_n$ in turn is just the limit of Grassmann varieties
\[
\BGL_n:=\Gr(\infty,n):=\colim_{N\to\infty}\Gr(N,n)
\]
where $\Gr(N,n)$ is the Grassmannian of rank $n$ quotients of $\sO^N$, and $U_n\to \BGL_n$ is the representing rank $n$ bundle with universal quotient $\pi_{N,n}:\sO^N\to U_n$. The inductive system $\ldots\to \Gr(N,n)\to\Gr(N+1,n)\to\ldots$ is defined via the projections
\[
\sO^{N+1}\to\sO^N
\]
on the first $N$ factors.

We have the closed immersion $i_n:\BGL_n\to \BGL_{n+1}$ representing the surjection 
\[
\sO^N\oplus\sO\xrightarrow{\pi_n\oplus\id}U_n\oplus \sO
\]
on $\BGL_n$; the gluing maps $\epsilon_n:\MGL_n\wedge T\to\MGL_{n+1}$ are just the maps
\[
\MGL_n\wedge T=\Th(U_n)\wedge T=\Th(U_n\oplus\sO)\xrightarrow{\Th(i_n)}\Th(U_{n+1})=\MGL_{n+1}
\]

The resulting bi-graded cohomology theory $X\mapsto\MGL$ is an {\em oriented} theory. There are many equivalent definitions of this notion; for details we refer the reader to \cite{Panin,PaninPimenovRoendigs}. For our purposes, we can take an oriented theory to be one with a good theory of first Chern classes for line bundles. In the case of $\MGL$, $c_1(L)$ is defined as follows. Take a line bundle $L\to X$, with $X$ smooth and quasi-projective over $k$.  By Jouanoulou's trick, and the homotopy invariance of $\MGL^{*,*}$, we can assume that $X$ is affine, and hence $L$ is generated by global sections. Thus, $L$ is the pull-back of $\sO_{\P^N}(1)$ for some morphism $f:X\to \P^N$. Modulo checking independence of various choices, this reduces us to defining $c_1(\sO(1))\in \MGL^{2,1}(\P^\infty)$, where $\MGL^{2,1}(\P^\infty)$ is short-hand for $\lim_N \MGL^{2,1}(\P^N)$, the limit begin defined by a fixed sequence of linear embeddings
\[
\P^1\to \P^2\to\ldots\to \P^N\to\ldots\ .
\]

For this, note that $\BGL_1=\P^\infty$, $U_1\to\BGL_1$ is $\sO(1)$, and hence $\MGL_1=Th(\sO_{\P^\infty}(1))$. The bonding maps in $\MGL$ give us the sequence of maps
\[
\MGL_1\wedge T^{\wedge n}\to \MGL_{n+1}.
\]
This defines the map  in $\Spt(k)$
\[
\iota:\Sigma_T^\infty Th(\sO_{\P^\infty}(1))\to S^{2,1}\wedge \MGL,
\]
giving us the class $[\iota]\in\MGL^{2,1}(Th(\sO_{\P^\infty}(1)))$. Composing the 0-section $s:\P^\infty\to \sO_{\P^\infty}(1)$ with the canonical quotient map $\sO_{\P^\infty}(1)\to Th(\sO_{\P^\infty}(1))$ defines $\pi:\P^\infty\to Th(\sO_{\P^\infty}(1))$; we set
\[
c_1(\sO(1)):=\pi^*([\iota])\in \MGL^{2,1}(\P^\infty).
\]

An orientation on a bi-graded cohomology gives rise to a good theory of push-forward maps for projective morphisms (see \cite[theorem 4.1.4]{Panin} for a detailed statement). For $\MGL$, this says we have functorial push-forward maps
\[
f_*:\MGL^{a,b}(Y)\to \MGL^{a+2d,b+d}(X)
\]
for each projective morphism $f:Y\to X$ in $\Sm/k$, where $d=\codim f:=\dim_kX-\dim_kY$. The connection with the first Chern map is that, for $L\to X$ a line bundle on $X\in\Sm/k$ with zero-section $s:X\to L$, one has
\[
c_1(L)=s^*(s_*(1_X))
\]
where $1_X\in\MGL^{0,0}(X)$ is the unit.

In fact, our extension \cite[theorem 1.10]{LevineOrient} of Panin's theorem gives good projective push-forward maps for $\MGL$-cohomology with supports. We describe the general situation.

 Let $\SP$ denote the category of {\em smooth pairs}, this being the category with objects $(M,X)$, $M\in \Sm/k$, $X\subset M$ a closed subset. A morphism $f:(M,X)\to (N,Y)$ is a morphism $f:M\to N$ in $\Sm/k$ such that $f^{-1}(Y)\subset X$. We have as well the category $\SP'$, with the same objects as $\SP$, but where a morphism $f:(N,Y)\to (M,X)$ is a projective morphism $f:N\to M$ in $\Sm/k$ with $f(Y)\subset X$. 
 
 For any $T$-spectrum $E$, sending $(M,X)\in\SP$ to the {\em $E$-cohomology with supports}
\[
E^{a,b}_X(M):=\Hom_{\SH(k)}(\Sigma^\infty_TM/(M\setminus X),\Sigma^{a,b}E)
\]
defines a functor $E^{*,*}$ from $\SP^\op$ to bi-graded abelian groups. In case $E$ is an oriented ring $T$-spectrum, $(M,X)\mapsto E^{*,*}_X(M)$ defines a bi-graded  {\em oriented ring cohomology theory on $\SP$}, in the sense of \cite[definition 1.3]{LevineOrient}. In particular, for $E=\MGL$, we have the bi-graded oriented ring cohomology theory $(M,X)\mapsto \MGL^{*,*}_X(M)$ on $\SP$.

Let $E$ be a a bi-graded  oriented ring cohomology theory on $\SP$. By \cite[theorem 1.10]{LevineOrient}, there are push-forward maps 
\[
f_*:E^{a,b}_Y(N)\to E^{a+2d,b+d}_X(M)
\]
for each map $f:(N,Y)\to (M,X)$ in $\SP'$, where $d=\codim f$, such that the maps $f_*$ define an 
{\em integration with supports} on $E^{*,*}$, in the sense of \cite[definition 1.6]{LevineOrient}. Without listing all of this definition here, this means that $f\mapsto f_*$ is functorial on $\SP'$, satisfies a projection formula with respect to cup products, and is compatible with pull-back in transverse cartesian squares. In addition, the maps $f_*$ are compatible with the boundary maps in the long exact sequence of triples.

Next, we extend the integration on the oriented theory $E^{*,*}$ to a bi-graded {\em oriented duality theory} $(H,E)$. Let $\Sch_k'$ be the category of reduced quasi-projective schemes over $k$, with morphisms the projective morphisms. For $X\in\Sch_k'$, choose a closed immersion $X\to M$ with $M\in \Sm/k$, and define
\[
H_{a,b}(X):=E^{2d_M-a,d_M-b}_X(M)
\]
where $d_M:=\dim_kM$ (we assume that $M$ is equi-dimensional over $k$).  In \cite[theorem 3.4]{LevineOrient}, we show that $X\mapsto H_{*,*}(X)$ extends to a functor $H_{a,b}:\Sch_k'\to \Ab$, and that the pair $(H_{*,*}, E^{*,*})$ defines a bi-graded oriented duality theory on $\Sch_k$. We refer the reader to \cite[defintion 3.1]{LevineOrient} for the precise definition of this notion, noting that this includes comparison isomorphisms
\[
\alpha_{M,X}:H_{*,*}(X)\to E^{2d_M-*,d_M-*}_X(M)
\]
for each closed immersion $X\to M$, $M\in\Sm/k$, such that, if we are given a projective morphism $f:Y\to X$, closed immersions $Y\to N$, $X\to M$, $N,M\in\Sm/k$, and an extension of $f$ to a projective morphism $F:N\to M$, then the diagram
\[
\xymatrix{
H(Y)\ar[r]^{\alpha_{N,Y}}\ar[d]_{f_*}&A_Y(N)\ar[d]^{F_*}\\
H(X)\ar[r]_{\alpha_{M,X}}&A_X(M)
}
\]
commutes. In addition, $H_{*,*}$ has pull-back maps for open immersions, a boundary map
\[
\delta_{X,Y}:H_{a,b}(X\setminus Y)\to H_{a-1,b}(Y)
\]
for each closed subset $Y\subset X$,  external products
\[
H_{*,*}(X)\otimes H_{*,*}(Y)\to H_{*,*}(X\times Y)
\]
and cap products
\[
f^*(-)\cap : A^{a,b}_X(M)\otimes H_{p,q}(Y)\to H_{p-a,q-b}(f^{-1}(X))
\]
for each map $f:Y\to M$ and smooth pair $(M,X)\in\SP$, such that these operations are compatible with the corresponding ones on $E$-cohomology with supports via the comparison isomorphisms $\alpha$.  

Let $L\to Y$ be a line bundle on some $Y\in\Sch_k$. Using the fact there is a line bundle $\sL\to M$ for some $M\in \Sm/k$, and a morphism $f:Y\to M$ with $L\cong f^*\sL$, the  cap product with  $c_1(\sL)$ gives a well-defined first Chern class operator
\[
\tilde{c}_1(L):H_{p,q}(Y)\to H_{p+2,q+1}(Y),
\]
independent of the choice of smooth envelope $Y\to M$ and extension of $L$ to $\sL$.  

\begin{rem} One can think of a bi-graded oriented duality theory $(H,E)$ as a generalization of a Bloch-Ogus twisted duality theory, the difference being that one does not require that $c_1(L\otimes M)=c_1(L)+c_1(M)$.   Replacing this is the {\em formal group law} $F_E(u,v)\in E^{2*,*}(k)[[u,v]]$, $E^{2*,*}(k):=\oplus_nE^{2n,n}(k)$. This is the power series characterized by the identity
\[
F_E(c_1(L),c_1(M))=c_1(L\otimes M)
\]
for each pair of line bundles $L,M$ on a fixed $X\in\Sm/k$. See \cite[\S3.9]{Panin} for further details.
\end{rem}

We denote the extension of $\MGL^{*,*}$ to a bi-graded oriented duality theory by $(\MGL'_{*,*}, \MGL^{*,*})$.

\section{$\Omega_*$ and $\MGL'_{2*,*}$}\label{sec:OmegaEtc} Together with F. Morel \cite{LevineMorel}, we have defined the ``geometric" theory of algebraic cobordism $\Omega_*$, on $\Sch_k$. By \cite[theorem 7.1.3]{LevineMorel}, the theory $\Omega_*$ is the universal oriented Borel-Moore homology theory on $\Sch_k$, in the sense of \cite[definition 5.1.3]{LevineMorel}. In particular, for each $n$, $\Omega_n$ is a functor 
\[
\Omega_n:\Sch_k'\to \Ab,
\]
for each \lci morphism $g:X'\to X$ of relative dimension $d$, there is a pull-back map
\[
g^*:\Omega_n(X)\to\Omega_{n+d}(X'),
\]
 functorial in $g$,  
there is an associative and commutative external product
\[
\times:\Omega_n(X)\otimes\Omega_{n'}(X)\to \Omega_{n+n'}(X\times X'),
\]
and a unit element $1\in\Omega_0(k)$.

The external product makes $\Omega_*(k)$ a commutative, graded ring, and $\Omega_*(X)$ a graded $\Omega_*(k)$-module for each $X$. In fact, there is a canonical isomorphism $\L_*\cong \Omega_*(k)$, where $\L_*$ is the Lazard ring, giving us the formal group law $F_\Omega(u,v)\in\Omega_*(k)[[u,v]]$ (see \cite[theorem 4.3.7]{LevineMorel}).

Define the first Chern class operator of a line bundle $L\to X$ with zero-section $s:X\to L$ as
\[
\tilde{c}_1(L)(\alpha):=s^*(s_*(\alpha)).
\]
Then the locally nilpotent operators $\tilde{c}_1(L):\Omega_*(X)\to \Omega_{*-1}(X)$ commute with one another (for fixed $X$) and  satisfy the formal group law
\[
F_\Omega(\tilde{c}_1(L), \tilde{c}_1(M))=\tilde{c}_1(L\otimes M).
\]

For $X\in\Sch_k$ and $n\ge0$ an integer, let $\sM_n(X)$ be the free abelian group on the set of isomorphism classes of projective maps $f:Y\to X$, with $Y\in\Sm/k$, and $Y$ irreducible of dimension $n$ over $k$. Sending $X$ to $\sM_n(X)$ becomes a functor
\[
\sM_n:\Sch'_k\to \Ab
\]
where, for $g:X\to X'$ a projective morphism and $f:Y\to X$ in $\sM_n(X)$, we define $g_*(f:Y\to X):=g\circ f:Y\to X'$. Additionally, for $g:X'\to X$ a smooth, quasi-projective morphism of relative dimension $d$, we define
\[
g^*:\sM_n(X)\to \sM_{n+d}(X')
\]
by $g^*(f:Y\to X):=p_2:Y\times_XX'\to X'$. Finally, we have an external product
\[
\times:\sM_n(X)\otimes\sM_{n'}(X')\to \sM_{n+n'}(X\times X')
\]
by sending $(f:Y\to X)\otimes(f':Y'\to X')$ to $f\times f':Y\times Y'\to X\times X'$ (strictly speaking, we take the sum of the restrictions of $f\times f'$ to the irreducible components of $Y\times Y'$). 

The construction of $\Omega_*$ gives a natural surjection
\[
\rho_X:\sM_n(X)\to \Omega_n(X)
\]
compatible with push-forward $g_*$ for projective $g$, pull-back $g^*$ for smooth, quasi-projective $g$ and the external product $\times$. 

Now let $(H,E)$ be an oriented duality theory. For each $Y\in\Sm/k$, we have the unit $1_Y\in E^{0,0}(Y)$. If $Y$ has dimension $n$ over $k$,  the comparison isomorphism
\[
\alpha_{Y,Y}:H_{2n,n}(Y)\to E^{0,0}(Y)
\]
gives us  the {\em fundamental class} $[Y]_H\in H_{2n,n}(Y)$
\[[Y]_H:=\alpha_{Y,Y}^{-1}(1_Y).
\]
For $X\in\Sch_k$, we may map $\sM_n(X)$ to $H_{2n,n}(X)$ by sending $f:Y\to X$ to $f_*([Y]_H)$; by \cite[proposition 4.2]{LevineOrient}, this descends to a homomorphism
\[
\vartheta_H(X):\Omega_*(X)\to H_{2*,*}(X),
\]
natural with respect to projective push-forward, pull-back by open immersions and compatible with external products and first Chern class operators. By \cite[lemma 4.3]{LevineOrient}, the natural transformation $\vartheta_H$, restricted to $\Sm/k$, defines a natural transformation of oriented cohomology theories
(in the sense of \cite[definition 1.1.2]{LevineMorel})
\[
\vartheta^E:\Omega^*\to E^{2*,*};\quad \vartheta^E(X):=\alpha_{X,X}\circ\vartheta_H(X).
\]
Here $\Omega^*(Y):=\Omega_{n-*}(Y)$ where $n=\dim_kY$.

We can now state our main result.

\begin{thm}\label{thm:main} Let $k$ be a field of characteristic zero. Then
\[
\vartheta_{\MGL'}(X):\Omega_*(X)\to \MGL'_{2*,*}(X)
\]
is an isomorphism for all $X\in\Sch_k$.
\end{thm}

The proof with occupy the next three sections.
\section{$\MGL^{*,*}$ for fields} \label{sec:Fields}

Let $X$ be a smooth irreducible scheme over $k$, $F=k(X)$. By \cite[theorem 4.3.7]{LevineMorel}, $\Omega^*(F)=\L^{*}:=\L_{-*}$. Thus, the natural transformation $\vartheta_\MGL:\Omega^*\to \MGL^{2*,*}$ gives a ring homomorphism
\[
\phi:\L^*\to \MGL^{2*,*}(F).
\]

Next, we want to define a group homomorphism
\[
\psi_X:\Gamma(X,\sO_X^*)\to \MGL^{1,1}(X).
\]
For this,  we note that, in the  homotopy category $\sH(k)$, $B\G_m$ represents the Picard functor on $\Sm/k$, and similarly, in the pointed homotopy category $\sH_\bullet(k)$, $B\G_m$ represents the relative Picard functor on pairs of smooth varieties (this follows directly from \cite[proposition 4.3.8]{MorelVoevodsky}). Since
\[
\Gamma(X,\sO_X^*)\cong \Pic(X\times\A^1,X\times\{0,1\}),
\]
it suffices to construct a map
\[
\psi:B\G_m\to \MGL_1.
\]
But $B\G_m\cong\P^\infty$ in $\sH_\bullet(k)$, and $\MGL_1=Th(\sO_{\P^\infty}(1))$, so the composition
\[
\P^\infty\xrightarrow{\text{zero-section}}\sO_{\P^\infty}(1)\xrightarrow{\pi}Th(\sO_{\P^\infty}(1))
\]
does the trick.

\begin{rem} The careful reader with note that the map $\psi$ does not send the base-point of $B\G_m=\P^\infty$ to the base-point of $\MGL_1$. We correct this by extending $\psi$ to $\tilde{\psi}:\P^\infty\cup_{(1:0:\ldots)}\A^1\to \MGL_1$, by identifying $\A^1$ with the fiber of $\sO(1)$ over $(1:0:\ldots)$. Using the base-point $1\in\A^1$ for 
$\P^\infty\cup_{(1:0:\ldots)}\A^1$ makes $\tilde{\psi}$ a pointed map, and the collapse map
\[
(\P^\infty\cup_{(1:0:\ldots)}\A^1, 1)\to(\P^\infty,(1:0:\ldots)) 
\]
is an isomorphism in $\sH_\bullet(k)$.
\end{rem}

Let $\HZ\in\SH(k)$ be the $T$-spectrum classifying motivic cohomology (see e.g. 
 \cite{OstRond, PaninPimenovRoendigs, Voevodsky}). We have the canonical map
\[
\rho_\HZ:\MGL\to \HZ
\]
classifying $\HZ$ as an oriented cohomology theory (\cite[theorem 1.0.1]{PaninPimenovRoendigs}).

\begin{lem}\label{lem:units} For $X\in\Sm/k$, the composition
\[
\rho_\HZ(X)\circ\psi_X:\Gamma(X,\sO_X^*)\to H^{1,1}(X)
\]
is an isomorphism of abelian groups.
\end{lem}

\begin{proof} By \cite[theorem 3.2.4(1)]{Panin}, the map $\psi:B\G_m\to \MGL_1$ composed with the canonical map
\[
\iota_1:\Sigma^\infty_T\MGL_1\to S^{2,1}\wedge \MGL
\]
induces the first Chern class map for the oriented theory $\MGL$ via the natural transformation
\begin{multline*}
\Pic(X,A)\cong \Hom_{\sH_\bullet(k)}(X/A,B\G_m)\xrightarrow{\psi_*}\Hom_{\sH(k)}(X/A,\MGL_1)\\
\xrightarrow{\iota_*}\Hom_{\SH(k)}(\Sigma_T^\infty(X/A),S^{2,1}\wedge \MGL).
\end{multline*}
As $\rho_\HZ$ induces a map of oriented theories, $\rho_\HZ$ is compatible with the respective Chern classes. Thus, we see that the composition
\begin{multline*}
\Pic(X,A)\cong \Hom_{\sH_\bullet(k)}(X/A,B\G_m)\xrightarrow{\psi_*}\Hom_{\sH_\bullet(k)}(X/A,\MGL_1)\\
\xrightarrow{\iota_*}\Hom_{\SH(k)}(\Sigma_T^\infty(X/A),S^{2,1}\wedge \MGL)\xrightarrow{\rho_\HZ} \Hom_{\SH(k)}(\Sigma_T^\infty(X/A),S^{2,1}\wedge \HZ)
\end{multline*}
induces the first Chern class map for motivic cohomology. Since
\[
c_1^\HZ:\Gamma(X,\sO_X^*)=\Pic(X\times\A^1,X\times\{0,1\})\to
H^{2,1}(X\times\A^1/X\times\{0,1\})=H^{1,1}(X)
\]
is an isomorphism, the result follows.
\end{proof}

Using $\phi$ and the $\MGL^{*,*}(k)$-module structure on $\MGL^{*,*}(X)$, $\psi_X$ extends to
\[
\psi_X:\Gamma(X,\sO_X^*)\otimes\L^*\to \MGL^{2*+1,*+1}(X).
\]

\begin{lem}\label{lem:surj} Let $X\in\Sm/k$ be irreducible. For $F=k(X)$, the map
\[
\phi_F:\L^*\to \MGL^{2*,*}(F)
\]
is an isomorphism, and the map
\[
\psi_F:F^\times\otimes \L^*\to \MGL^{2*+1,*+1}(F)
\]
is surjective.
\end{lem}

\begin{proof} We  use the Hopkins-Morel spectral sequence for $\Spec F$
\begin{equation}\label{eqn:HMSS}
E^{p,q}_2(n)=\oplus_nH^{p-q,n-q}(F)\otimes\L^{q}\Longrightarrow \MGL^{p+q,n}(F).
\end{equation}
Since $H^{a,b}(F)=0$  if $a>b$ or $b<0$,  $H^{0,0}(F)=\Z$, $H^{1,1}(F)=F^\times$, and $\L^q=0$ if $q>0$, this gives us
\[
E^{p,q}_2(n)= \begin{cases} 
0&\text{ if }p>n\\
\L^{n}&\text{ if }p=q=n\\
F^\times\otimes L^{n-1}&\text{ if }p=n, q=n-1.
\end{cases}
\]
Thus, elements of $E^{n,n}_2(n)$ and $E^{n,n-1}(n)$ are permanent cycles, giving us surjections
\[
E^{n,n}_2(n)\to E^{n,n}_\infty(n);\quad E^{n,n-1}_2(n)\to E^{n,n-1}_\infty(n).
\]
In addition, if $p+q=2n$ or $p+q=2n-1$, the only non-zero $E_2^{p,q}$ terms are $E^{n,n}_2$ or $E^{n,n-1}_2$, so we have
\[
\MGL^{2n,n}(F)=E^{n,n}_\infty(n);\ \MGL^{2n-1,n}(F)=E^{n,n-1}_\infty.
\]
It follows directly from the construction of the spectral sequence that the composition
\[
\L^n=E^{n,n}_2(n)\to E^{n,n}_\infty(n)\to\MGL^{2n,n}(F)
\]
is $\phi_F$. Similarly, the composition
\[
F^\times\otimes \L^{n-1}=E^{n,n-1}_2(n)\to E^{n,n-1}_\infty(n)\to \MGL^{2n-1,n}(F)
\]
is $\psi_F$. This proves that $\phi_F$ and $\psi_F$ are surjective.

To see that $\phi_F$ is injective, we may assume that $k$ admits an embedding $\sigma:k\to \C$. Via this embedding, we have the oriented cohomology theory $\MU^{*,*}_\sigma$ on $\Sm/k$ 
\[
\MU^{a,b}_\sigma(Y):= \MU^a(Y(\C)).
\]
By \cite{PaninPimenovRoendigs}, this induces a natural transformation
\[
\MGL^{a,b}(Y)\to \MU^a(Y(\C))
\]
of oriented cohomology theories. Note that
\[
\MU^{a,b}_\sigma(F)=\colim_U\MU^a(U(\C))
\]
as $U$ runs over non-empty Zariski open subsets of $X$.

The composition
\[
\L^*\to \MGL^{2*,*}(k)\to\MU^{2*,*}_\sigma(k)= \MU^{2*}(pt)
\]
is the ring homomorphism classifying the formal group law of $\MU^*$. As this is the universal formal group law, this composition is an isomorphism. Now let $U\subset X$ be a non-empty open subset. Since $U$ has a $\C$-point, the pull-back map
\[
\MU^*(pt)\to \MU^*(U(\C))
\]
is injective, and thus
\[
\MU_\sigma^{*,*}(pt)\to \MU^{*,*}_\sigma(F)
\]
is injective as well. This shows that  $\phi_F$ is injective, completing the proof.
\end{proof}

\section{Chern classes, divisors and the boundary map}\label{sec:Boundary}

Let $U$ be an open subscheme of some $X\in\Sm/k$ with closed complement $i:D\to X$, and inclusion $j:U\to X$. Given $u\in \Gamma(U,\sO^*)$, we have the element
\[
\psi(u)\in\MGL^{1,1}(U).
\]
Taking the boundary in the long exact localization sequence
\[
\ldots\to \MGL^{1,1}(X)\xrightarrow{j^*}\MGL^{1,1}(U)\xrightarrow{\partial_{X,D}}\MGL^{2,1}_D(X)\xrightarrow{i_*}\MGL^{2,1}(X)\to\ldots
\]
gives us the element
\[
\partial_{X,D}(\psi(u))\in \MGL^{2,1}_D(X).
\]
Our goal in this section is to give a formula for $\partial(\psi(u))$. 

To help with our computation, we introduce the {\em effective   Picard monoid with supports}, $\Pic^\eff_D(X)$, and the relative first Chern class
\[
c_1^D(L)\in \MGL^{2,1}_D(X).
\]
The set $\Pic^\eff_D(X)$ is the set of isomorphism classes of pairs $(L,s)$, with $L\to X$ a line bundle, and $s:X\to L$ a section that is nowhere vanishing on $\X\setminus D$; an isomorphism of such pairs $\phi:(L,s)\to (L',s')$ is an isomorphism of line bundles $\phi:L\to L'$ with $s'=\phi\circ s$. We make  $\Pic^\eff_D(X)$ a monoid using tensor product:
\[
(L,s)\cdot(L',s'):=(L\otimes L',s\otimes s');
\]
the unit is $(\sO_X,1)$. In fact, the same formula defines a commutative product
\[
\cdot:\Pic^\eff_{D_1}(X,D_1)\times \Pic^\eff_{D_2}(X)\to \Pic^\eff_{D_1\cup D_2}(X).
\]
\begin{rem} This notion is taken from Fulton \cite{Fulton}, who calls a pair $(L,s)$ a {\em pseudo-divisor}.
\end{rem}

Let $\Pic^\eff(X)\subset \Pic(X)$ be the monoid of isomorphism classes of line bundles $L$ on $X$ that admit a non-zero section. We have the evident forgetful maps
\[
\Pic^\eff_D(X)\to \Pic^\eff(X)\to \Pic(X).
\]
We now describe how to lift the Chern class map
\[
c_1:\Pic(X)\to \MGL^{2,1}(X)
\]
to a Chern class map with supports
\[
c_1^D:\Pic^\eff_D(X)\to \MGL^{2,1}_D(X).
\]

Given $(L,s)\in\Pic^\eff_D(X)$, we first suppose that $L$ is generated by global sections on $X$. Extend $s$ to a finite set of generating sections
\[
s=s_0,s_1,\ldots, s_N,
\]
giving us the map
\[
f:X\to \P^N.
\]
Define
\[
h:U\times\A^1\to \P^N
\]
by 
\[
h(u,t):=(s_0(u):ts_1(u):\ldots:ts_N(u));
\]
since $s_0$ is nowhere zero on $U$, $h$ is well-defined. Also
\[
h(u,1)=f(u);\quad h(u,0)=(1:0:\ldots:0),
\]
and thus $f\cup h$ gives a well-defined map of pointed spaces over $k$:
\[
f\cup h:X\cup_{U\times 1}U\times\A^1/U\times 0\to (\P^N,(1:0:\ldots:0)).
\]
Composing $f\cup h$ with
\[
\P^N\to \P^\infty=\BGL_1\xrightarrow{\psi}Th(\sO(1))=\MGL_1
\]
and noting that the collapse map
\[
X\cup_{U\times 1}U\times\A^1/U\times 0\to X/U
\]
is an isomorphism in $\sH_\bullet(k)$ gives us
\[
c_1^D(L)\in\MGL^{2,1}(X/U)=\MGL^{2,1}_D(X).
\]
It is easy to check that the map $f\cup h$ is independent (as a map in $\sH_\bullet(k)$) of the choices made. Thus, we have a well-defined element $c_1^D(L)\in\MGL^{2,1}_D(X)$, assuming that $L$ is generated by global sections.

In general, we use Jouanoulou's trick. Take an affine space bundle $p:Y\to X$ with $Y$ affine, and replace $(X,D,L)$ with $(Y,p^{-1}(D),p^*(L))$. As $Y$ is affine, $p^*(L)$ is generated by global sections, so we can apply the construction of the preceding paragraph, noting that
\[
p^*:\MGL^{*,*}_D(X)\to\MGL^{*,*}_{p^{-1}(D)}(Y)
\]
is an isomorphism. Since the collection of such affine space bundles forms a directed system, it is easy to check that the resulting class 
\[
c_1^D(L):=(p^*)^{-1}(c_1^{p^{-1}(D)}(p^*(L)))
\]
is independent of the choice of $p:Y\to X$, completing the construction.

Let $\SP^1$ be the full subcategory of the category $\SP$ of smooth pairs consisting of $(X,D)$ with $D\subset X$ a pure codimension one closed subset. Clearly, $(X,D)\mapsto \Pic^\eff_D(X)$ defines a functor from $\SP^1$ to the category of monoids, and 
\[
c_1^D:\Pic^\eff_D(X)\to\MGL^{2,1}_D(X)
\]
 is a natural transformation.  Similarly, it is easy to see that the diagram
\begin{equation}\label{eqn:ChernSupp}
\xymatrix{
\Pic^\eff_D(X)\ar[r]\ar[d]_{c_1^D}&\Pic(X)\ar[d]^{c_1}\\
\MGL^{2,1}_D(X)\ar[r]&\MGL^{2,1}(X)
}
\end{equation}
commutes.

\begin{ex}\label{ex:main} Take $X=\A^1=\Spec k[x]$, $D=0$, $(L,s)=(\sO_{\A^1},x)$. Note that $\A^1/\A^1\setminus0\cong(\P^1,(1:0))$ in $\sH_\bullet(k)$, so
\[
\MGL^{2,1}_0(\A^1)\cong\MGL^{2,1}(\P^1,(1:0))\cong\MGL^{0,0}(k)=\Z.
\]
Then $c_1^0(\sO_{\A^1},x)=\pm1$. One can verify this by a direct computation. Alternatively, we can use the naturality of $c_1^0$, replacing $(\A^1, 0,(\sO_{\A^1},x))$ with $(\P^1,(1:0),(\sO(1),X_1))$; by excision, the restriction map
\[
\MGL^{2,1}_{(1:0)}(\P^1)\to \MGL^{2,1}_0(\A^1)
\]
is an isomorphism. Using the commutativity of the diagram \eqref{eqn:ChernSupp}, we see that $c_1^{(1:0)}(\sO(1),X_1)$ maps to $c_1(\sO(1))\in \MGL^{2,1}(\P^1)$. By the projective bundle formula, 
\[
\MGL^{2,1}(\P^1)=c_1(\sO(1))\cdot \MGL^{0,0}(k)\oplus\MGL^{2,1}(k)=c_1(\sO(1))\cdot \MGL^{0,0}(k),
\]
so $c_1(\sO(1))$ is a generator of $\MGL^{2,1}(\P^1)=\Z$, as desired. If we normalize the isomorphism 
\[
\MGL^{2,1}_0(\A^1)\cong \MGL^{2,1}(\P^1)\cong \Z
\]
by using $c_1(\sO(1))$ as the distinguished generator, then  $c_1^0(\sO_{\A^1},x)=+1$.
\end{ex}

We now show that the Chern class with supports computes the boundary map in our localization sequence.

\begin{lem}\label{lem:BoundaryComp} Let $f:X\to\A^1$ be a dominant morphism, with $X\in\Sm/k$ irreducible, let $D=f^{-1}(0)$ and $U=X\setminus D$. Let $u\in\Gamma(U,\sO^*)$ be the restriction of $f$. Then
\[
\partial_{X,D}(\psi(u))=c_1^D(\sO_X,f)\in\MGL^{2,1}_D(X).
\]
\end{lem}

\begin{proof} Both sides of the equations are natural in $(X,f)$, so it suffices to handle the universal case $X=\A^1$, $f=\id$. Writing $\A^1=\Spec k[x]$, $u$ is just the canonical unit $x$ on $U=\A^1\setminus\{0\}$. The boundary map
\[
\partial:\MGL^{1,1}(U)\to\MGL^{2,1}_0(\A^1)
\]
is induced by the map $\delta$ in the Puppe sequence
\[
U\to \A^1\to \A^1\cup_{U\times1}U\times\A^1/U\times0\xrightarrow{\delta}(\A^1/\{0,1\})\wedge U_+
\]
Noting that $\psi(u)$ descends canonically to an element in reduced cohomology
\[
\bar\psi(u)\in \MGL^{1,1}(U,1)=\MGL^{1,1}(\G_m),
\]
we can use the pointed version of the Puppe sequence
\[
\G_m\to (\A^1,1)\to \A^1\cup_{U\times1}U\times\A^1/(U\times0\cup 1\times\A^1)\xrightarrow{\bar\delta}(\A^1/\{0,1\})\wedge \G_m
\]
Then $\bar\delta$ is an isomorphism in $\sH_\bullet(k)$, and both terms are isomorphic to $(\P^1,(1:0))$. The map
\[
\Sigma_T^\infty (\A^1/\{0,1\})\wedge \G_m\to S^{2,1}\wedge\MGL
\]
representing $\psi(u)$ is by definition the map induced by the canonical inclusion
\[
\P^1\to\P^\infty
\]
followed by the map $\P^\infty\to\MGL_1$ induced by the zero-section $\P^\infty\to\sO(1)$, and pre-composed with the canonical isomorphism
\[
(\A^1/\{0,1\})\wedge \G_m\cong (\P^1,(1:0))
\]
in $\sH_\bullet$. But we have already seen that this gives us $c_1(\sO(1))\in\MGL^{2,1}(\P^1)$, which is the same as $c_1^0\in\MGL^{2,1}_0(\A^1)$.
\end{proof}

Suppose the $D\subset X$ is the support of a strict normal crossing divisor, that is, if $D$ has irreducible components $D_1,\ldots, D_m$, then for each  $I\subset\{1,\ldots, m\}$, the intersection
\[
D_I:=\cap_{i\in I}D_i
\]
is smooth and has codimension $|I|$ on $X$; we call the $D_I$ the {\em strata} of $D$. For $(L,s)\in\Pic^\eff_D(X)$, let $\Div(s)=\sum_in_iD_i$ denote the usual divisor, that is, $n_i$ is the order of vanishing of $s$ along $D_i$. We have defined in \cite[definition 3.1.5]{LevineMorel} the {\em cobordism divisor class}
\[
\Div^\Omega(s)\in\Omega_{\dim X-1}(D).
\]
with the following properties:
\begin{enumerate}
\item Let $i:D\to X$ be the inclusion. Then
\[
i_*(\Div^\Omega(s))=c_1(L):=\tilde{c}_1(L)(1_X)\in\Omega_{\dim X-1}(X).
\]
\item Write $D=D_1\cup\ldots\cup D_r$, with each $D_i$ a smooth codimension one closed subscheme of $X$, such that $\Div(s)$ can be written as $\Div(s)=\sum_{i=1}^rn_iD_i$. Let  $\xi_i= \tilde{c}_1(\sO_X(D_i))$. Then for each non-empty $I\subset\{1,\ldots, r\}$ there are universal power series 
\[
G_I(u_1,\ldots, u_r)\in \Omega_*(k)[[u_1,\ldots, u_r]]
\]
such that
\[
\Div^\Omega(s)=\sum_I\iota_{I*}[G_I(\xi_1,\ldots\xi_r)([D_I])],
\]
where $D_I=\cap_{i\in I}D_i$,  $\iota_I:D_I\to D$ is the inclusion, and $[D_I]\in\Omega_*(D_I)$ is the class of $\id_{D_I}$. 
\end{enumerate}

\begin{rem} Let $X, D,s,L$ be as above. In \cite{LevineMorel}, we used the notation $[\Div(s)\to D]\in\Omega_*(D)$ for $\Div^\Omega(s)$.
\end{rem}

\begin{lem}\label{lem:DivNat} Let $(X,D)$ be in $\SP$, such that $D$ is a reduced strict normal crossing divisor on $X$. Let $f:X'\to X$ be a morphism in $\Sm/k$. Suppose that $f$ is transverse to the inclusion $D_I\to X$ for each stratum $D_I$ of $D$. Let $(H,E)$ be a bi-graded oriented duality theory on $\Sch_k$, $\vartheta_H:\Omega_*\to H_{2*,*}$ the natural transformation given by \cite[proposition 4.2]{LevineOrient}. Letting $D'=f^{-1}(D)$, we have the maps
\begin{align*}
&\alpha_{X,D}\circ \vartheta_H(D):\Omega_{d_X-1}(D)\to E^{2,1}_D(X)\\
&\alpha_{X',D'}\circ \vartheta_H(D'):\Omega_{d_{X'}-1}(D')\to E^{2,1}_{D'}(X')\\
&f^*:E^{2,1}_D(X)\to E^{2,1}_{D'}(X')
\end{align*}
Then for $(L,s)\in \Pic^\eff_D(X)$, we have $(f^*(L),f^*(s))\in \Pic^\eff_{D'}(X')$ and
\[
f^*(\alpha_{X,D}\circ \vartheta_H(D)(\Div^\Omega(s)))=\alpha_{X',D'}\circ \vartheta_H(D')(\Div^\Omega(f^*(s)))).
\]
\end{lem}

\begin{proof} Write $D=D_1\cup\ldots \cup D_r$ with each $D_i$ a smooth codimension one closed subscheme of $X$, such that $\Div(s)$ can be written as $\Div(s)=\sum_{i=1}^rn_iD_i$.  Let $D_i'=f^{-1}(D_i)$. Then  $D'$ is a reduced  strict normal crossing divisor on $X'$, 
$D'=D'_1\cup\ldots\cup D'_r$, each $D'_i$ is a smooth codimension one closed subscheme of $X'$ (or is empty), and $\Div(f^*(s))=\sum_in_iD'_i$. Thus, letting $\xi_i=\tilde{c}_1(\sO_X(D_i))$, $\xi_i'=\tilde{c}_1(\sO_{X'}(D_i'))$, $s'=f^*(s))$, we have
\begin{align}
&\Div^\Omega(s)=\sum_I\iota_{I*}[G_I(\xi_1,\ldots\xi_r)([D_I])]\label{align:1}\\
&\Div^\Omega(s')=\sum_I\iota_{I*}[G_I(\xi'_1,\ldots\xi'_r)([D'_I])].\label{align:2}
\end{align}
The maps $\alpha_{X,D}\circ \vartheta_H(D)$ are natural with respect to projective push-forward and commute with the respect first Chern class operators, hence
\[
\alpha_{X,D}\circ \vartheta_H(D)(\iota_{I*}[G_I(\xi_1,\ldots\xi_r)([D_I])])=
\iota_{I*}[G_I(\xi_1,\ldots\xi_r)(\alpha_{X,D_I}\circ \vartheta_H(D_I)([D_I]))]
\]
and similarly for $X', D', s'$. Furthermore, since $D_I$ and $D'_I$ are smooth, we have
\[
\iota_{I*}[G_I(\xi_1,\ldots\xi_r)(\alpha_{X,D_I}\circ \vartheta_H(D_I)([D_I]))]=
\iota_{I*}[G_I(\xi_1,\ldots\xi_r)(\vartheta^E(D_I)(1^\Omega_{D_I}))]
\]
and similarly for $D'_I$. Here $1^\Omega_{D_I}\in\Omega^*(D_I)$ is the unit. Finally, since 
$\vartheta^E$ is a natural transformation of oriented cohomology theories on $\Sm/k$, we have
$\vartheta^E(D_I)(1^\Omega_{D_I})=1^E_{D_I}$, and similarly for $D'_I$. Putting these identities together gives
\begin{align}
&\alpha_{X,D}\circ \vartheta_H(D)(\iota_{I*}[G_I(\xi_1,\ldots\xi_r)([D_I])])
=\iota_{I*}[G_I(\xi_1,\ldots\xi_r)(1_{D_I})]\label{align:3}\\
&\alpha_{X',D'}\circ \vartheta_H(D')(\iota_{I*}[G_I(\xi'_1,\ldots\xi'_r)([D'_I])])
=\iota_{I*}[G_I(\xi'_1,\ldots\xi'_r)(1_{D'_I})].\label{align:4}
\end{align}

Let $f_I:D'_I\to D_I$ be the restriction of $f$. Since the diagram
\[
\xymatrix{
D'_I\ar[d]_{\iota'_I}\ar[r]^{f_I}&D_I\ar[d]^{\iota_I}\\
X'\ar[r]_f&X
}
\]
is transverse, the diagram
\[
\xymatrix{
E^{a,b}(D'_I)\ar[d]_{\iota'_{I*}}&E^{a,b}(D_I)\ar[l]_{f_I^*}\ar[d]^{\iota_{I*}}\\
E^{2|I|+a,|I|+b}_{D'}(X')&E^{2|I|+a,|I|+b}_D(X)\ar[l]^{f^*}
}
\]
commutes (see \cite[lemma 1.7]{LevineOrient}). Similarly, $f_I^*$  commutes with the respective first Chern class operators, so
\[
f_I^*\circ G_I(\xi_1,\ldots\xi_r)=G_I(\xi'_1,\ldots\xi'_r)\circ f_I^*.
\]
Finally, $f_I^*:E^{*,*}(D_I)|to E^{*,*}(D'_I)$ is a ring homomorphism, so
\[
f_I^*(1_{D_I})=1_{D'_I}.
\]
Together with the identities \eqref{align:1}, \eqref{align:2}, \eqref{align:3}, \eqref{align:4}, this completes the proof.
\end{proof}

\begin{lem}\label{lem:ChenClassDiv} Take $X\in\Sm/k$, $D$ a strict normal crossing divisor on $X$ and $(L,s)\in\Pic^\eff_D(X)$. Then
\[
c_1^D(L,s)=\alpha_{X,D}(\vartheta _{\MGL'}(D)(\Div^\Omega(s)),
\]
where $\alpha_{X,D}:\MGL_{2\dim X-2,\dim X-1}(D)\to \MGL^{2,1}_{D}(X)$ is the comparison isomorphism.
\end{lem}

\begin{proof} With the help of lemma~\ref{lem:DivNat}, both sides are natural with respect to morphisms that are transverse to all the strata of $D$. Thus,   we may use Jouanoulou's trick to reduce to the case in which $X$ is affine. Write $D$ as a union of irreducible components
\[
D=D_1\cup\ldots\cup D_m
\]
and write $\Div(s)=\sum_in_iD_i$, $n_i\ge0$. Letting $L_i=\sO_X(D_i)$, we therefore have
\[
L\cong L_1^{\otimes n_1}\otimes\ldots\otimes L_m^{\otimes n_m}
\]
In addition, there are sections $s^{(i)}$ of $L_i$, and a unit $u$ on $X$ such that the divisor of $s_i$ (as a cycle on $X$) is $D_i$ and with
\[
s=u\cdot (s^{(1)}_1)^{n_1}\otimes\ldots\otimes (s^{(m)})^{n_m}.
\]
As $(L,s)\cong (L,u^{-1}s)$ and $\Div^\Omega(s)=\Div^\Omega(u^{-1}s)$, we may assume $u=1$. As $X$ is affine, each $L_i$ is generated by global sections, and we can find a set of generating sections of $L_i$ of the form
\[
s^{(i)}=s^{(i)}_0,\ldots, s^{(i)}_N.
\]
Thus, we have a morphism
\[
g:X\to\prod_{i=1}^m\P^N
\]
such that $g^*(p_i^*(\sO(1)))\cong L_i$ and with $g^*(X_0^{(i)})=s^{(i)}$. Here $p_i:\prod_{i=1}^m\P^N\to\P^N$ is the projection on the $i$th factor,  $X_0,\ldots, X_N$ are the standard coordinates on $\P^N$, and $X^{(i)}_j:=p_i^*(X_j)$. 

Let $H_i\subset \prod_{i=1}^m\P^N$ denote the subscheme defined by $X_0^{(i)}=0$, and $H=H_1\cup\ldots\cup H_m$. Since $g$ is transverse to all the strata of $H$, we may use lemma~\ref{lem:DivNat} to reduce us to the case $X=\prod_{i=1}^m\P^N$, $D=H$, $s=\prod_{i=1}^m(X_0^{(i)})^{n_i}$ and $L=\bigotimes_{i=1}^mp_i^*(\sO(1))^{\otimes n_i}$. In this case, the natural map
\[
\MGL^{2*,*}_H(\prod_{i=1}^m\P^N)\to \MGL^{2*,*}(\prod_{i=1}^m\P^N)
\]
is injective (see \cite[lemma 5.2.11]{LevineMorel}), so it suffices to see that
\[
c_1^\MGL(\bigotimes_{i=1}^mp_i^*(\sO(1))^{\otimes n_i})=\vartheta ^\MGL(i_{H*}(\Div^\Omega(s)))\in \MGL^{2,1}(\prod_{i=1}^m\P^N).
\]
But by \cite[proposition 3.1.9]{LevineMorel}
\[
c_1^\Omega(\bigotimes_{i=1}^mp_i^*(\sO(1))^{\otimes n_i})= i_{H*}(\Div^\Omega(s))\in \Omega^1(\prod_{i=1}^m\P^N).
\]
Since $\vartheta ^\MGL$ is compatible with the respective Chern class maps, this completes the proof.
\end{proof}

\begin{prop} \label{prop:Boundary} Let $f:X\to\P^1$ be a dominant morphism with $X\in\Sm/k$. Let $D_0=f^{-1}(0)$, $D_\infty=f^{-1}(\infty)$, $D=D_0\amalg D_\infty$, $U=X\setminus D$, $u=f_{|U}\in\Gamma(U,\sO^*)$. Write $\partial_{X,D}(u)\in\MGL^{2,1}_D(X)$ as a sum
\[
\partial_{X,D}(u):=\partial^0_{X,D}(u)-\partial^\infty_{X,D}(u);\quad \partial^0_{X,D}(u)\in\MGL^{2,1}_{D_0}(X), 
 \partial^\infty_{X,D}(u)\in\MGL^{2,1}_{D_\infty}(X), 
\]
according to the canonical direct sum decomposition
\[
\MGL^{2,1}_D(X)=\MGL^{2,1}_{D_0}(X)\oplus \MGL^{2,1}_{D_\infty}(X).
\]
Let $f_0=f^*(X_1)\in\Gamma(X,f^*\sO(1))$ and $f_\infty=f^*(X_0)\in \Gamma(X,f^*\sO(1))$. Suppose that $D$ is a strict normal crossing divisor. Then
\begin{align*}
&\partial^0_{X,D}(\psi(u))=\alpha_{X,D_0}(\vartheta _{\MGL'}(D_0)(\Div^\Omega(f_0))\\
&\partial^\infty_{X,D}(\psi(u))=\alpha_{X,D_\infty}(\vartheta _{\MGL'}(D_\infty)(\Div^\Omega(f_\infty)).
\end{align*}
\end{prop}

\begin{proof} We first verify the formula for $\partial^0_{X,D}(\psi(u))$. By excision, we can replace $X$ with $X\setminus D_\infty$, so we may assume that $f$ is a regular function on $X$, $f:X\to\A^1$. The formula for $\partial^0_{X,D}(\psi(u))$ then follows from lemma~\ref{lem:BoundaryComp}  and lemma~\ref{lem:ChenClassDiv}.

The formula for $\partial^\infty_{X,D}(\psi(u))$ follows from the formula for $\partial^0_{X,D}(\psi(u))$. Indeed, $\psi$ and 
$\partial_{X\setminus D_0,D_\infty}$ are group homomorphisms, so
\[
\partial^\infty_{X,D}(\psi(u))=-\partial_{X\setminus D_0, D_\infty}(\psi(u))=\partial_{X\setminus D_0,D_\infty}(\psi(u^{-1})).
\]
However, if we replace $u$ with $u^{-1}$, this switches the roles of $D_0$ and $D_\infty$, and of $f_0$ and $f_\infty$. Thus the case we have already handled shows that
\[
\partial_{X\setminus D_0,D_\infty}(\psi(u^{-1}))=\alpha_{X,D_\infty}(\vartheta _{\MGL'}(D_\infty)(\Div^\Omega(f_\infty)),
\]
which completes the proof.
\end{proof}

\section{Proof of the main theorem}\label{sec:Proof} We are now ready to prove theorem~\ref{thm:main}. We proceed by induction on the maximal dimension of an irreducible component of $X$; the case of dimension 0 has been verified in lemma~\ref{lem:surj}. So, write  $X$ as a union of closed subsets
\[
X=X_1\cup\ldots\cup X_r\cup X'
\]
with $X_1,\ldots, X_r$ irreducible of dimension $d$ and $X'$ of dimension $<d$. Let $X_{(d)}=x_1\amalg\ldots\amalg x_r$ be the dimension $d$ generic points of $X$. Set
\[
\Omega_*^{(1)}(X):=\colim_W\Omega_*(W)
\]
as $W\subset X$ runs over the closed subsets containing no $x_i$ and set
\[
\MGL_{2*,*}^{(1)}(X):=\colim_W\MGL'_{2*,*}(W)
\]
over the same system of $W$. Taking the limit of the respective localization sequences gives us the commutative diagram with exact columns:
\[
\xymatrixcolsep{20pt}
\xymatrix{
&\oplus_i\MGL'_{2*+1,*}(k(x_i))\ar[d]^\partial\\
\Omega_*^{(1)}(X)\ar[r]^-{\vartheta^{(1)}(X)}\ar[d]_{i_*}&
\MGL^{(1)}_{2*,*}(X)\ar[d]^{i_*}\\
\Omega_*(X)\ar[r]^-{\vartheta(X)}\ar[d]_{j^*}&
\MGL'_{2*,*}(X)\ar[d]^{j^*}\\
\oplus_i\Omega_*(k(x_i))\ar[r]_-{\oplus_i\vartheta(x_i)}\ar[d]&
\oplus_i\MGL'_{2*,*}(k(x_i))\ar[d]\\
0&0
}
\]

We have already seen (lemma~\ref{lem:surj}) that the maps $\vartheta(x_i)$ are isomorphisms; the map 
$\vartheta^{(1)}(X)$ is an isomorphism by our induction hypothesis. This already implies that
$\vartheta(X)$ is surjective.

To show that $\vartheta(X)$ is injective, let $\Z[k(x_i)^\times]$ be the free abelian group on $k(x_i)^\times$. By lemma~\ref{lem:surj}, we have a surjection
\[
\psi_i:\Z[k(x_i)^\times]\otimes \L_*\to \MGL_{2*+2d-1, *+d-1}'(k(x_i)).
\]
Thus, it suffices to define for each $i$ a map
\[
\Div^\Omega_i:\Z[k(x_i)^\times]\otimes \L_*\to \Omega_{*+d-1}^{(1)}(X)
\]
making the diagram 
\begin{equation}\label{eqn:MainDiag}
\xymatrix{
\Z[k(x_i)^\times]\otimes \L_*\ar[r]^-{\alpha^{-1}_{k(x_i)}\circ\psi_i}\ar[d]_{\Div^\Omega_i}& \MGL_{2*+2d-1, *+d-1}'(k(x_i))\ar[d]^{\partial_i}\\
\Omega_{*+d-1}^{(1)}(X)\ar[r]_-{\vartheta^{(1)}(X)}&\MGL^{(1)}_{2*+2d-2,*+d-1}(X)
}
\end{equation}
commute, and with
\[
i_*\circ \Div^\Omega_i=0.
\]
Since, for $a\in k(x_i)^\times$, $\alpha\in \MGL'_{2d-2,d-1}(k(x_i))=\MGL^{1,1}(k(x_i))$,  $b\in\L_*$, we have
\[
\psi_i(a\otimes b)=\psi_i(a)\cup b; \partial_i(\alpha\cup b)=\partial(\alpha)\cup b,
\]
it suffices to define $\Div^\Omega_i$ on $k(x_i)^\times$ so that the diagram commutes and with $i_*\circ\Div^\Omega_i=0$ on $k(x_i)$; we then simply extend by linearity and by using the $\L_*$-module structure.

Take $u\in k(x_i)^\times$. Fix a blow-up $\pi:\tilde{X}_i\to X$  of $X_i$ so that
\begin{enumerate}
\item $\tilde{X}_i$ is smooth.
\item $u$ defines a morphism $f:\tilde{X}_i\to\P^1$.
\item $D_0:=f^{-1}(0)$ and $D_\infty:=f^{-1}(\infty)$ are strict normal crossing divisors.
\end{enumerate}

We use the notation of proposition~\ref{prop:Boundary}. Let $f_0=f^*(X_1)$, $f_\infty=f^*(X_0)$ and define
\[
\Div^\Omega_i(u):=\pi_*(\Div^\Omega(f_0)-\Div^\Omega(f_\infty))\in \Omega^{(1)}_{d-1}(X).
\]
Since $\Div^\Omega_i$ is defined on $\Z[k(x_i)^\times]$, we need not check that $\Div^\Omega_i(u)$ is independent of the choice of $\tilde{X}_i$.

We first show that
\[
i_*\circ\Div^\Omega_i(u)=0.
\]
Indeed, let $D=D_0\cup D_\infty$ and let $\tilde{\iota}:D\to \tilde{X}_i$ be the inclusion. Then by \cite[proposition 3.1.9]{LevineMorel}
\[
\tilde{i}_*(\Div^\Omega(f_0))=c_1(f^*(\sO(1))=\tilde{i}_*(\Div^\Omega(f_\infty)),
\]
so $\tilde{i}_*(\Div^\Omega(f_0)-\Div^\Omega(f_\infty))=0$. Since
\[
\pi_*\circ\tilde{i}_*=i_*\circ \pi_*
\]
it follows that $i_*\circ\Div^\Omega_i(u)=0$, as desired.

Finally, we check that the diagram \eqref{eqn:MainDiag} commutes. The boundary map in the localization sequence for $\MGL'_{*,*}$ is compatible with the boundary map in the Gysin sequence for $\MGL^{*,*}$-theory with supports, i.e.,
\[
\partial^\MGL_{\tilde{X}_i,D}\circ\alpha_{\tilde{X}_i\setminus D}=\alpha_{\tilde{X}_i,D}\circ \partial^{\MGL'}_{\tilde{X}_i,D}.
\]
Thus, proposition~\ref{prop:Boundary} gives us
\begin{equation}\label{eqn:Boundary2}
\partial^{\MGL'}_{\tilde{X}_i,D}(\alpha_U^{-1}(\psi(u)))=\vartheta(D)(\Div^\Omega(f_0)-\Div^\Omega(f_\infty)).
\end{equation}

The boundary map in the localization sequence for $\MGL'_{*,*}$ is natural with respect to projective push-forward (this follows from \cite[lemma 2.6]{LevineOrient}), i.e.,
\[
\partial_i\circ\pi_*=\pi_*\circ\partial_{\tilde{X}_i,D}.
\]
Thus, applying $\pi_*$ to \eqref{eqn:Boundary2} yields
\begin{align*}
\partial_i(\alpha_U^{-1}(\psi(u)))&=\pi_*(\partial^{\MGL'}_{\tilde{X}_i,D}(\alpha_U^{-1}(\psi(u)))\\
&=\pi_*(\vartheta(D)(\Div^\Omega(f_0)-\Div^\Omega(f_\infty)))\\
&=\vartheta^{(1)}(X)(\pi_*(\Div^\Omega(f_0)-\Div^\Omega(f_\infty)))\\
&=\vartheta^{(1)}(X)(\Div^\Omega_i(u)),
\end{align*}
as desired. This verifies the commutativity of the diagram  \eqref{eqn:MainDiag} and completes the proof of theorem~\ref{thm:main}.


\begin{thebibliography}{99}
\bibitem{Fulton}
Fulton, William. {\bf Intersection theory}. Ergebnisse der Mathematik
und ihrer Grenzgebiete (3), {\bf 2}. Springer-Verlag, Berlin-New York, 1984. xi+470.

\bibitem{LevineOrient} Levine, M. 
{\sl Oriented cohomology, Borel-Moore homology and algebraic cobordism},
preprint 2008.\ http://www.math.neu.edu/${}^\sim$levine/publ/Publ.html


\bibitem{LevineMorel}
Levine, M.; Morel, F. {\bf Algebraic Cobordism} Springer 2007.

\bibitem{Morel} 
Morel, F. {\sl  An introduction to $\mathbb A\sp 1$-homotopy theory}.  {\em Contemporary developments in algebraic $K$-theory} 357--441, ICTP Lect. Notes, XV, Abdus Salam Int. Cent. Theoret. Phys., Trieste, 2004.

\bibitem{MorelVoevodsky}  Morel, F. and Voevodsky, V., {\sl $\A^1$-homotopy theory of schemes},
Inst. Hautes \'Etudes Sci. Publ. Math. { 90} (1999), 45--143.


\bibitem{OstRond} 
{\O}stv{\ae}r, P.A.  and R\"ondigs, O. {\sl Motives and modules over motivic cohomology}. C. R. Math. Acad. Sci. Paris {\bf 342} (2006), no. 10, 751--754.

\bibitem{Panin0} 
Panin, I. {\sl Oriented cohomology theories of algebraic varieties}.  
Special issue in honor of Hyman Bass on his seventieth birthday. Part III. 
$K$-Theory {\bf 30} (2003), no. 3, 265--314. 

\bibitem{Panin} I. Panin {\sl Push-forwards in oriented cohomology theories of algebraic varieties II}. Preprint 2003. http://www.math.uiuc.edu/K-theory/0619/

\bibitem{PaninPimenovRoendigs} Panin, I., Pimenov, K. and R\"ondigs, O.
{\sl A universality theorem for Voevodsky's algebraic cobordism spectrum}. Preprint (2007). $K$-theory preprint archive. http://www.math.uiuc.edu/K-theory/0846/

\bibitem{Voevodsky}
V. Voevodsky, {\sl $\A^1$-homotopy theory}, Proceedings of the International 
Congress of Mathematicians, Vol. I (Berlin, 1998). Doc. Math. 1998, 
Extra Vol. I, 579--604.
\end{thebibliography}
\end{document}